\documentclass[10pt,journal]{amsart}
\usepackage{verbatim}
\usepackage{amsmath}
\usepackage{amssymb}
\usepackage{amsfonts}
\usepackage{latexsym}
\usepackage{amscd}
\usepackage{amsthm}
\usepackage{graphicx}
\usepackage{xcolor}
\usepackage{ulem}

\newcommand{\Ac}{\mathcal{A}}
\newcommand{\Nc}{\mathcal{N}}

\newcommand{\Z}{\mathbb{Z}}

\newcommand{\st}{{\textrm{ : }}}
\newcommand{\<}{\langle}
\renewcommand{\>}{\rangle}
\newcommand{\qd}{\mathcal{QD}^{-1}}
\newcommand{\qdx}{{\mathcal{QD}}^{+1}}

\newtheorem{thm}{Theorem}
\newtheorem{lem}[thm]{Lemma}
\newtheorem{cor}[thm]{Corollary}
\newtheorem{rem}[thm]{Remark}
\newtheorem{pro}[thm]{Proposition}
\theoremstyle{definition}
\newtheorem{defn}[thm]{Definition}

\begin{document}
\title{On abelian Group Representability  of finite Groups}
\author{
Eldho K. Thomas, Nadya Markin,  Fr\'ed\'erique Oggier}
\address{Division of Mathematical Sciences, School of Physical and Mathematical Sciences, \newline\indent Nanyang Technological University, Singapore.} 
\email {ELDHO1@e.ntu.edu.sg,NMarkin@ntu.edu.sg,frederique@ntu.edu.sg}

\thanks{E. K. Thomas was supported by Nanyang Technological University under Research 
Grant M58110049, N. Markin was supported by the Singapore National
Research Foundation under Research Grant NRF-CRP2-2007-03, F. Oggier was partially funded by 
Nanyang Technological University under Research Grant M58110049.}

\maketitle
\begin{abstract}
A set of quasi-uniform random variables $X_1,\ldots,X_n$ may be generated from a finite group $G$ and $n$ of its subgroups, with the corresponding entropic vector depending on the subgroup structure of $G$. 
It is known that the set of entropic vectors obtained by considering arbitrary finite groups is much richer than the one provided just by  abelian groups. In this paper, we start to investigate in more detail different families of non-abelian groups with respect to the entropic vectors they yield. In particular, we address the question of whether a given non-abelian group $G$ and some fixed subgroups $G_1,\ldots,G_n$ end up giving the same entropic vector as some abelian group $A$ with subgroups $A_1,\ldots,A_n$, in which case we say that 
$(A, A_1, \ldots, A_n)$ represents $(G, G_1, \ldots, G_n)$. If for any choice of subgroups $G_1,\ldots,G_n$, there exists some abelian group $A$ which represents $G$, we refer to $G$ as being abelian (group) representable for $n$. We completely characterize dihedral, quasi-dihedral  and dicyclic groups with respect to their abelian representability, as well as the case when $n=2$, for which we show a group is abelian representable if and only if it is nilpotent. This problem is motivated by understanding non-linear coding strategies for network coding, and network information theory capacity regions.    
\end{abstract}
%
%
%

\section{Introduction}

Let $X_1, \ldots, X_n$ be a collection of $n$ jointly distributed discrete random variables over some alphabet of size $N$.
We denote by $\Ac$ a subset of indices from $\Nc = \{1,\ldots ,n\}$, and $X_\Ac=\{X_i,~i\in\Ac\}$.
The {\it entropic vector} corresponding to $X_1,\ldots,X_n$ is the vector
\[
h=(H(X_1),\ldots,H(X_1,X_2),\ldots H(X_1,\ldots,X_n))
\]
which collects all the joint entropies $H(X_\Ac)$, $\Ac\subseteq \Nc$.
Consider the $(2^n-1)$-dimensional Euclidean space, with coordinates labeled by all possible non-empty subsets $\Ac \subseteq \{1, \ldots, n\}$.
The region formed by all entropic vectors $h$ is denoted by $\Gamma_n^*$. Determining $\Gamma_n^*$ is notoriously difficult. However, it is also the key to understanding the capacity region of any arbitrary multi-source multi-sink acyclic network, which can be computed as a linear function of the entropy vector over $\bar{\Gamma}_n^*$, under linear constraints \cite{YYZ,HS}.
 
It is known \cite{chan} that the closure $\bar{\Gamma}_n^*$ of the set of all entropic vectors is equal to the convex closure (the smallest closed convex set containing the set) of all entropic vectors generated by quasi-uniform random variables.

\begin{defn}
A probability distribution over a set of $n$ random variables $X_1,\ldots,X_n$ is said to be {\it quasi-uniform} if for any $\Ac\subseteq\Nc$, $X_\Ac$ is uniformly distributed
over its support $\lambda(X_\Ac)$:
\[
P(X_\Ac=x_\Ac)=
\left\{
\begin{array}{ll}
1/|\lambda(X_\Ac)| & \mbox{if }x_\Ac\in\lambda(X_\Ac), \\
0 & \mbox{otherwise}.
\end{array}
\right.
\]
\end{defn}

By definition, the entropy of a quasi-uniform distribution is
\[
H(X_\Ac) = -\sum_{x_\Ac\in\lambda(X_\Ac)} P(X_\Ac=x_\Ac)\log P(X_\Ac=x_\Ac)   
         = \log |\lambda(X_\Ac)|. 
\]
We now turn to the question of characterizing  quasi-uniform random variables. Let $G$ be a finite group with $n$ subgroups $G_1,\ldots,G_n$, and let $X$ be a random variable uniformly distributed over $G$, that is, $P(X=g)=1/|G|$ for any $g\in G$.
Define a new random variable $X_i=XG_i$, whose support is $[G:G_i]$ cosets of $G_i$ in $G$. Then
$P(X_i=gG_i)=|G_i|/|G|$ since $X_i=gG_i$ whenever $g\in G_i$, and 
$P(X_i=gG_i,~i\in\Ac)=|\cap_{i\in\Ac}G_i|/|G|$.
This shows that quasi-uniform random variables may be obtained from finite groups, more precisely:

\begin{thm}\cite{chan,T}
\label{thm:three}
For any finite group $G$ and any subgroups $G_1, \ldots, G_n$ of $G$, there exists a set of $n$ jointly distributed quasi-uniform discrete random variables $X_1,\ldots, X_n$ such that for all non-empty subsets $\Ac$ of $\Nc$, $H(X_\Ac)= \log{[G:G_\Ac]}$,
where $G_\Ac=\cap_{i\in\Ac}G_i$ and $[G:G_\Ac]$ is the index of $G_\Ac$ in $G$.
\end{thm}

Thus, given a group $G$ and $n$ subgroups $G_1,\ldots,G_n$, one obtains $n$ quasi-uniform random variables. The subgroup structure of $G$ determines the correlation among the $n$ random variables, and in turn the corresponding entropic vector.

\begin{defn}
Let $X_1, \ldots, X_n$ be $n$ jointly distributed discrete random variables. The corresponding entropic vector $h$ is said to be {\it group representable}, if there exists a group $G$, with subgroups $G_1,\ldots, G_n$ such that $H(X_\mathcal{A})= \log[G:G_\mathcal{A}]$ for all $\mathcal{A}$.
If in addition, the group $G$ is abelian, we say that $h$ is  \textit {abelian group representable}.
\end{defn}

Interestingly, considering only quasi-uniform distributions coming from finite groups is in fact enough to understand $\bar{\Gamma}_n^*$:

\begin{thm}\cite{Y}
Let $\Upsilon_n$ be the region of all group representable entropic vectors, and let $\overline{con}(\Upsilon_n)$ be its convex closure. Then
\[
\overline{con}(\Upsilon_n) ~=~ \bar{\Gamma}_n^*.
\]
\end{thm}

It is also shown in \cite{chan} that the collection of all abelian group representable entropic vectors $\Upsilon_n^{ab}$  gives a non-trivial inner bound for $\Gamma_n^*$ when $n$ is at least 4.

\begin{thm} \cite{chan}
When $n\ge 4, ~\overline{con}\Upsilon_n^{ab} \ne \bar{\Gamma_n^*}$.
\end{thm}

Thus entropic vectors which are abelian group representable are a proper subset of entropic vectors coming from quasi-uniform random variables. This addresses a natural question of classifying groups with respect to the entropic vectors which they induce. In particular, we want to understand which groups belong to the same class as abelian groups with respect to this classification. We make this precise in the definitions below.

\begin{defn}
Let $G$ be a group and let $G_1, \ldots, G_n$ be fixed subgroups of $G$.
Suppose there exists an abelian group $A$  with subgroups $A_1, \ldots, A_n$ such that for every non-empty $\Ac \subseteq\Nc$, $[G:G_\Ac] = [A: A_\Ac]$, where $G_\Ac = \cap_{i\in \Ac}G_i, A_\Ac = \cap_{i\in \Ac}A_i$. Then we say that
$(A, A_1, \dots, A_n)$ {\it{represents}} $(G, G_1, \ldots, G_n)$.
\end{defn}

\begin{defn}
If  for every choice of subgroups $G_1, \ldots, G_n$ of $G$, there exists an abelian group $A$ such that
$(A, A_1, \dots, A_n)$ {{represents}} $(G, G_1, \ldots, G_n)$, we say that $G$ {\it{is abelian (group) representable for $n$}}. 
\end{defn}

Note that the abelian group $A$ may vary for different choices of subgroups $G_1, \ldots, G_n$.
Sometimes, however, it is possible to find an abelian group that works for all choices  $G_1, \ldots,G_n$.

\begin{defn} Suppose there exists an abelian group $A$ such that for every choice of subgroups
$G_1, \ldots, G_n \leq G$, there exist subgroups $A_1, \ldots, A_n \leq A$ such that
 $(A, A_1, \ldots, A_n)$ represents $(G, G_1, \ldots, G_n)$. Then we will say that $G$ is {\it{uniformly abelian (group) representable}} for $n$. (Alternatively, $A$ {\it{uniformly represents}} $G$. )
\end{defn}

When $G$ is abelian group representable, the quasi-uniform random variables $X_1,\ldots, X_n$ corresponding to subgroups $G_i$ can also be obtained using an abelian group $A$ and its subgroups $A_1, \ldots, A_n$.
If we choose the subgroup $G_i=G$, then $\log [G:G]=0$, that is $H(X_i)=0$, which implies $X_i$ is actually taking values deterministically. If  $G_i=\{1\}$, then $\log [G:\{1\}]= \log |G|$. Thus the entropy chain
rule yields
\[
H(X_i,X_\Ac)=H(X_i)+H(X_\Ac|X_i)
\]
for every $\Ac$ such that $i\notin \Ac$. Since $H(X_i)=\log |G|$ and $H(X_i,X_\Ac)=\log [G: \{1\} \cap G_\Ac]=\log |G|$,
we conclude that
\[
H(X_\Ac|X_i)=0.
\]
That is, given $X_i$, all the $n-1$ other random variables  are  functions of $X_i$ and we will consequently require that each subgroup $G_1, \ldots G_n$ is non-trivial and proper. Hence $n$ is at most the number of non-trivial proper subgroups of $G$.

The contributions of this paper are to introduce the notion of abelian group representability for an arbitrary finite group $G$, and to characterize the abelian group representability of several classes of groups:
\begin{itemize}
\item
Dihedral, quasi-dihedral and dicylic groups are shown to be abelian group representable for every $n \geq 2$ if (Section \ref{sec:2-group}) and only if (Section \ref{sec:nilpotent}) they are $2$-groups. 
When they are abelian group representable, they are furthermore uniformly abelian group representable for every $n$ (Section \ref{sec:2-group}).
\item
$p$-groups: $p$-groups are shown to be uniformly abelian group representable for $n=2$ in Section \ref{sec:p-groups}. 
\item
Nilpotent groups: in Section \ref{sec:nilpotent} we show that representability of nilpotent groups is completely determined by representability of $p$-groups. The set of nilpotent groups is shown to contain the set of abelian representable groups for any $n$ in Section \ref{sec:nilpotent}, the two   coincide  for $n=2$. 

\end{itemize}

%
%

\section{abelian Group Representability of classes of $2$-groups}
\label{sec:2-group}

In this section we establish uniform abelian group representability of dihedral, quasi-dihedral and dicyclic $2$-groups for any $n$.
We begin with a general lemma showing how abelian group representability of a group $H$ may imply abelian group representability of a group $G$.

\begin{lem}

Let $$\psi: G \rightarrow H$$ be a bijective map, which is additionally {\it{subgroup preserving}}, i.e., for any subgroup $G_i \leq G$, the set $\psi(G_i)$ is a subgroup of $H$. Suppose that $H$ is abelian (resp. uniformly abelian) group representable. Then so is $G$. In particular, if $H$ itself is abelian, then $G$ is abelian group representable.
\label{lem:psilemma}
\end{lem}

\begin{proof}
We want to show that given subgroups $G_1, \ldots, G_n \leq G$,  there exists an abelian group $A$ with subgroups $A_1, \ldots, A_n$ such that for any  subset $\Ac \subseteq \{1, \ldots, n\}$ the intersection subgroup $G_\Ac$ has index $[G:G_\Ac] =[A:A_\Ac]$.

Since $H$ is abelian group representable, $(H, \psi(G_1), \ldots \psi(G_n))$ can be represented by some $(A, A_1, \ldots, A_n)$. We claim that $(A, A_1, \ldots, A_n)$ will also represent $(G, G_1, \ldots G_n)$.

Since $\psi$ is bijective, for any $G_i \leq G$, the subgroup $\psi(G_i)$ has the same size and index in $H$ as $G_i$ has in $G$. In particular, $[A:A_i]=[H:\psi(G_i)]  = [G:G_i]$. This takes care of 1-intersection, i.e., when $|\Ac| = 1$.

Now we want to show that in fact $[G:G_\Ac] = [A:A_\Ac]$ for any $\Ac \subseteq \Nc$.
When considering intersections $G_\Ac$, let us first consider $2$-intersections $G_{12}= G_1 \cap G_2 \leq G$.
First observe that $\psi(G_1 \cap G_2) = \psi(G_1 )\cap( G_2)$. The containment $\psi(G_1 \cap G_2) \subseteq  \psi(G_1 )\cap\psi( G_2)$ is immediate. To see the containment $\psi(G_1 \cap G_2) \supseteq   \psi(G_1 )\cap( G_2)$ observe that if $y = \psi(g_1) = \psi(g_2)$ for some $g_1 \in G_1, g_2 \in G_2$ then bijectivity of $\psi$  implies that $g_1 = g_2$ and $y \in \psi(G_1\cap G_2)$.

Now, recalling that $|G| = |H|$,  we see that
$$[G: G_1\cap G_2] =  [H: \psi(G_1)\cap \psi(G_2)] = [A: A_1 \cap A_2].$$
More generally for arbitrary intersection $G_\Ac$, we have $[G:G_\Ac] = [A:A_\Ac]$:  this follows by induction on the number of subgroups involved in the intersection.
We conclude that $(A, A_i)_{i\in \Nc}$ represents $(G, G_i)_{i\in \Nc}$. If $H$ was uniformly abelian group representable, then $A$ was chosen independently of subgroups $\psi(G_1), \ldots , \psi(G_n)$ and it follows that $G$ is also uniformly abelian group representable. 

\end{proof}

\subsection{Dihedral and Quasi-dihedral $2$-Groups}

We will establish the abelian representability for dihedral and quasi-dihedral 2-groups. 
We define the dihedral group $D_{m}$ for $m \geq 3$ to be the symmetry group of the regular $m$-sided polygon. The group $D_m$ is of order $2m$, with a well known description in terms of generators and relations:
\[
D_{m}=\langle r,s|r^m=s^2=1,rs=sr^{-1}\rangle.
\]
Each element of $D_m$ is uniquely represented as $r^as^j$ where $1 \leq a \leq m$, $j=0,1$.

Note that the generator $s$ acts on $r$ by conjugation sending $r$ to an element $srs^{-1}= r^{-1}$. More generally, consider other possibilities for  $srs^{-1}=r^z$. When we apply this map twice, we see that since $s$ has order $2$, it must be that $z^2 \equiv 0 \mod 2^{m}$. In case $m=2^{k}$, there are $4$ such choices modulo $2^k$ for  $z \in \Z_{{2^{k}}}^\times$, i.e., $z \equiv \pm1, 2^{k-1} \pm1$. The choice $z=1$ will result in an abelian group, $z=-1$ gives the dihedral group $D_{2^k}$ above. Now we cover the remaining two choices. In either case, the subgroup structure is similar to that of a dihedral group, which will eventually allow us to conclude that these groups are abelian representable via Lemma \ref{lem:psilemma}.

Define two quasi-dihedral  groups, each of order $2^{k+1}$: 
$$ \qd_{2^k}= \langle r, s \mid r^{2^{k}}=s^2=1, rs = sr^{2^{k-1}-1} \rangle,$$
$$ \qdx_{2^k}= \langle r, s \mid r^{2^{k}}=s^2=1, rs = sr^{2^{k-1}+1} \rangle.$$

Let us make some brief observations about the structure of subgroups of \\
$D_{2^k}, \qd_{2^k}, \qdx_{2^k}$. 
First of all, note that a subgroup generated by $\langle r^{j_1}s, r^{j_2}s\rangle$ can also be generated by 
$   \langle r^{j_1}s, r^{j_2}s(r^{j_1}s)^{-1}\rangle = \langle r^{j_1}s, r^{j_2-j_1}\rangle$ which in  turn can be expressed as $  \langle  r^{2^i}, r^{j_1}s\rangle \text{ for some } i.$
We conclude that only one generator of the form $r^js$ is required. Trivially, only one generator of the form $r^j$ is required as well, since 
the cyclic subgroup $\langle r \rangle$ contains only cyclic subgroups $ \langle r^{2^i} \rangle$. 

Hence any subgroup can be expressed in terms of at most 2 generators of the form $r^{2^i}, r^jx$. When the subgroup is given by $\<r^{2^i}, r^jx\>$, without loss of generality we can further assume that $j < 2^i$: this is achieved by premultiplying the second generator $r^js$ repeatedly with $r^{-2^i}$. 
This gives us 3 basic subgroup types of dihedral and quasi-dihedral groups.
\begin{rem} \label{subtypesDqD}
Subgroups of $D_{2^k}, \qd_{2^k}, \qdx_{2^k}$ can be only of the following types: 
\begin{enumerate}
\item 
$
\langle r^{2^i}\rangle  		,  ~ 0 \leq i < k$,
\item 
$\langle r^a s\rangle    		, ~0\leq a <2^k$,
\item 
$\langle r^{2^i},r^cs\rangle		, ~0 \leq i < k, ~0\leq c <2^i.$ 
\end{enumerate}

\end{rem}

\subsubsection{Dihedral $2$-groups}
We now establish abelian group representability of dihedral $2$-groups for all $n$ by applying the lemma above. We start by showing that subgroups of $D_{2^k}$ map to subgroups of an abelian group.

\begin{pro} 
 Let $A = \oplus_{i=0}^{k} \Z_2$ be generated by the standard basis $\{e_0, \ldots e_k\}$ over integers $\Z_2$ modulo $2$. There exists a subgroup preserving bijection from the dihedral group $D_{2^k}$ to $A$.

\label{pro:psidihedral}
\end{pro}
\begin{proof}

For each element $r^a s^j$ we use the base $2$ representation of exponent $a = \sum_{i=0}^{k-1}a_i2^i$ to define the map $\psi: D_{2^k} \rightarrow A$:

$$\psi: r^a s^j \mapsto \sum_{i=0}^{k-1}a_ie_i + je_k.$$
We show that the map $\psi$ is a subgroup preserving bijection.

First note that since each element is uniquely represented as $r^as^j$ for some $a, j$, the map $\psi$ is indeed well-defined and clearly bijective. Hence we only need to verify  that  $\psi$ is subgroup preserving. In other words, if $H \leq D_{2^k}$ is a subgroup, then the image $\psi(H)$ is a subgroup of $A$. Since each element in $A$ is its own inverse, and $\psi(H)$ contains the identity, we only need to check the closure property of  $\psi(H)$ in order to prove that it is a subgroup.

To that end, we investigate the subgroups of $D_{2^k}$.  All subgroups of $D_{2^k}$ are of the form (see Remark \ref{subtypesDqD})

\begin{enumerate}
\item 
$
\langle r^{2^i}\rangle  		= \{r^a \st a \equiv  0 \mod 2^i\},  ~ 0 \leq i < k$,
\item 
$\langle r^a s\rangle    		=   \{1, r^a s\}, ~0\leq a <2^k$,
\item 
$\langle r^{2^i},r^cs\rangle    	= \{r^{a}, r^b s  \st a \equiv 0 \mod 2^i, b \equiv c \mod 2^i  \}, \\~0 \leq i < k, ~0\leq c <2^k.
$ 

\end{enumerate}


\noindent{\bf{Case 1}}: $H=\langle r^{2^i}\rangle   = \{r^a \st a \equiv  0 \mod 2^i\}$ . Note that $a < 2^k$ is a multiple of $2^i$ if and only if  $a = \sum_{j=i}^{k-1}a_j2^j$, i.e., the terms $a_0, \ldots a_{i-1}$ of the binary expansion of $a$ are $0$. In other words, the image

$$\psi(H) = \{ \sum_{j=i}^{k-1}a_je_j \}.$$ Clearly this set is closed under addition.

\noindent {\bf{Case 2}}: $H= \langle r^a s \rangle= \{1, r^a s\}$. Then the image  $$\psi(H) = \{\sum_{i=0}^{k-1}a_ie_i + e_k, 0\}.$$ Clearly this set of size 2 is a subgroup of $A$.\\

\noindent {\bf{Case 3}}: $H=\langle r^{2^i}, r^c s\rangle ,~ 0 \le c <2^i$.


Indeed $H  =  
\{r^{2^ih}, r^{c+2^ih}s \st  h \in \Z\}$.

We verify directly that $\psi(H)$ is closed under addition by showing that 

$$x, y \in H \implies \psi(x)+\psi(y) \in H.$$ This involves considering 3 cases: 
\begin{itemize}
\item
$x, y \in \<r^{2^i}\>$
\item 
$ x \in \<r^{2^i}\> , y = r^{c+2^ih }s$ 
\item 
both $x,y$ are of the form $r^{c+2^ih}s$. 
\end{itemize}
\begin{enumerate}

\item  
$x, y \in \<r^{2^i}\>$ implies that $\psi(x) + \psi(y) \in \psi(H).$ This follows identically to {\bf{Case 1}}. \\

\item 
$ x \in \<r^{2^i}\> , y = r^{c+2^ih' }s$ implies that $\psi(x) + \psi(y) \in \psi(H).$

We show that $\psi(r^{2^ih})+\psi(r^{2^ih'+c} s) = \psi(r^{2^i h'' +c  } s) \in   \psi(H)$.
To see this, observe using the binary expansion of exponents, that when $c < 2^i$ we can actually  factor $\psi(r^{2^i h'+c}) = \psi(r^{2^i h'})+\psi(r^c)$. Hence

$$\psi(r^{2^i h})+\psi(r^{2^i h' +c} s )  =
\psi(r^{2^i h}) + \psi(r^{2^i h'})+ \psi(r^c)+ \psi( s )  = $$
$$\psi(r^{2^i h''})+\psi(r^c)+\psi(s)= \psi(r^{2^i h'' +c}s) \in \psi(H).$$ \\

\item both $x,y$ of the form $r^{c+2^ih}s$ 
implies that $\psi(x) + \psi(y) \in \psi(H). $

Now $\psi(r^{2^ih+c} s)+\psi(r^{2^ih'+c} s) = \psi(r^{2^ih+c})+\psi(s)+\psi(r^{2^ih'+c} )+\psi(s) =$
$$ \psi(r^{2^ih})+\psi(r^c) +\psi(r^{2^ih'})+\psi(r^c) =
 \psi(r^{2^ih''}) \in \psi(H).$$

\end{enumerate}

\end{proof}

\begin{pro}
\label{pro:dig}
The dihedral group $D_{2^k}$ is uniformly abelian group representable for all $n$. \end{pro}
\begin{proof}
This follows from Proposition \ref{pro:psidihedral} and Lemma \ref{lem:psilemma}. The uniform part comes from the fact that $A$ only depends on the size of $D_{2^k}$ and not on the choice of subgroups of $D_{2^k}$.

\end{proof}

We next use abelian representability of dihedral groups to derive abelian representability of quasi-dihedral groups. The idea is that these two classes of groups have the same subgroups. 



\subsubsection{Quasi-dihedral $2$-Groups}


We already know by Remark \ref{subtypesDqD} what subgroups of quasi-dihedral groups look like in terms of generators. However, for the purpose of  defining a subgroup preserving map $\psi$ (which will not be a homomorphism!), we want to know exactly what elements these subgroups consist of.
The next lemma describes what the subgroups of quasi-dihedral groups look like. 

\begin{lem}
Let $G = \qd_{2^k} $ or $\qdx_{2^k}$ be a quasi-dihedral group of size $2^{k+1}$.Then all the subgroups  of $G$ are of the form 

\begin{enumerate}
\item $\langle r^{2^i}\rangle= \{r^{a} \st a \equiv 0 \mod 2^i\},  ~0\le i \le k-1 $,

\item $\langle r^{2^i},~r^js\rangle=  \{r^{a},r^{b}s \st a \equiv 0\mod 2^i, b \equiv j \mod 2^i\}$, $0\le i \leq k-1,~ 0\le j \le 2^i-1$, 

\item 
\begin{itemize}
\item When $G = \qd_{2^k}$, then 
$\langle r^js\rangle = 
\begin{cases} 
\{1, r^js\} & j \equiv 0 \mod 2 \\
\{1, r^js, r^{2^{k-1}},  r^{2^{k-1}+j}s\} & j \equiv 1 \mod 2 
\end{cases}
$

\item When $G = \qdx_{2^k}$, then  
$\langle r^js\rangle  =  
\begin{cases}
 \< r^{2j}, r^js \>  {\text{  is of type (2)}} & j \neq 0 \\
 \{1,s\} & j = 0.
 \end{cases}
$
\end{itemize} 
\end{enumerate}

\label{qdSubgroups}

\end{lem}

\begin{proof}
First assume that $G= \qd_{2^k}$. 
\begin{enumerate}
\item $\langle r^{2^i}\rangle= \{r^{a} \st a \equiv 0 \mod 2^i\} $

This case is obvious.

\item $\langle r^{2^i},~r^js\rangle=  \{r^{a},r^{b}s \st a \equiv 0\mod 2^i, b \equiv j \mod 2^i\}; $ {\\ \small{$ ~0 \leq  i \leq  k-1 , ~ 0\le j \le 2^i-1$ }}.
  
For this we observe that 
$r^js r^{2^i} = r^{j+2^i({2^{k-1}-1} ) }s $. 

Additionally note that 

$$r^js r^j s = r^{ j+j(2^{k-1}-1)} s^2 = r^{j(2^{k-1})} = 
\begin{cases}
1 & j \equiv  0 \mod 2 \\
r^{2^{k-1}} & j \equiv 1 \mod 2
\end{cases}$$ \item


 
$\langle r^js\rangle = 
\begin{cases} 
\{1, r^js\} & j \equiv 0 \mod 2 \\
\{1, r^js, r^{2^{k-1}},  r^{2^{k-1}+j}s\} & j \equiv 1 \mod 2 
\end{cases}
$

In particular when $j $ is odd, the subgroup $\<r^js  \>$ can be expressed in form (2) as $\<r^{2^{k-1}}, r^js\>$.
\end{enumerate}

Now let $G = \qdx_{2^k}$. 
The cases (1), (2) follow identically to $\qd_{2^k}$. 

For case (3), the case $j=0$ is obvious,  for $j > 0$, observe
$$r^jsr^js= r^jr^{j(2^{k-1}+1)}=r^{j(2^{k-1}+2)}=r^{2j(2^{k-2}+1)}.$$
But $\<r^{2j(2^{k-2}+1)}\> = \<r^{2j} \>$ and hence we have 

 $$\langle r^js\rangle = \langle r^{2j},~r^js\rangle. $$
Setting $2^i || 2j$ be the highest power of $2$ dividing $j$,  this subgroup is in fact of type (2) and consists of elements 
$ \{r^{a}, r^{j'}s \st a \equiv 0 \mod 2^i, j' \equiv j \mod 2^i\} .
$
\end{proof}

\begin{pro}
\label{pro:quasi}
The quasi dihedral group $\qd_{2^k}$ and  $\qdx_{2^k}$ are uniformly abelian group representable for all $n$.
\end{pro}

\begin{proof}
To prove the result we construct a subgroup-preserving bijection from quasi-dihedral to dihedral groups. By Lemma \ref{lem:psilemma} and the fact that dihedral groups are uniformly abelian group representable, the result follows. 

Let $G=\qd_{2^k}$ or $\qdx_{2^k}$. 

Elements of $G$  can be uniquely written as $$\{r^is^j \st 0 \leq i < 2^{k}, j=0,1\}.$$ 

Also all the elements in $D_{2^k}$ can be uniquely written as $$\{r^is^j \st 0 \leq i < 2^{k}, j=0,1\}, $$ keeping in mind  that $r, s \in D_{2^k}$ are not the same as $r, s \in G$, as the group law for the groups $G$ and $D_{2^k}$ is not the same. 

Define a subgroup preserving bijection $$\psi: G \rightarrow D_{2^k} $$ 
$$\psi: r^is^j \mapsto r^is^j. $$ 
The map $\psi $ is well-defined and clearly bijective. It remains to show that $\psi$ is subgroup preserving.

Now Lemma \ref{qdSubgroups} describes which elements subgroups of $G$ consist of. The proof of Proposition \ref{pro:psidihedral} gives a similar description for subgroups of $D_{2^k}$. Verifying that the two coincide via $\psi$ gives the result. As an example, consider a subgroup of form (2). We have

$$\psi(\langle r^{2^i},~r^js\rangle) = \psi(\{r^{a},r^{b}s \st a \equiv 0\mod 2^i, b \equiv j \mod 2^i\}) =$$
$$\{r^{a},r^{b}s \st a \equiv 0\mod 2^i, b \equiv j \mod 2^i\} \leq D_{2^k}. $$
Cases (1) and (3) follow immediately as well. 
Hence $\psi$ is a subgroup preserving bijection and since $D_{2^k}$ is abelian group representable by Proposition \ref{pro:dig}, Lemma \ref{lem:psilemma} implies that $\qd_{2^k}$ and $\qdx_{2^k}$ are uniformly abelian group representable as well.

\end{proof}

\subsection{Dicyclic $2$-Groups}


Next we consider the case of dicyclic groups, another well studied class of non-abelian groups. The results are similar to that of dihedral groups.
A dicyclic group $DiC_m$ of order $4m$ is generated by two elements $a, x$ as follows:
$$DiC_m = \langle a,x \st a^{2m}=1, x^2=a^m, xa= a^{-1}x \rangle .$$
Every element of $DiC_m$ can be uniquely presented as $a^ix^j$, where $0\leq i < 2m$, $j=0,1$.

{\it{A generalized quaternion group}} is a dicyclic group with $m$ a power of 2. We now study the subgroups of $DiC_{2^k}$. 
From the definition, we know that $|DiC_{2^k}|=2^{k+2}$. All the elements of $DiC_{2^k}$ are of the form $a^i, a^ix ,~ 1\le i\le 2^{k+1}$ where $x^2= a^{2^k}$ and $x^3= a^{2^k}x$.

As is the case with dihedral groups, any subgroup $\langle a^{j_1}x,a^{j_2}x \rangle = \langle a^{j_1-j_2}, a^{j_2}x \rangle$, which in turn can be represented as
 $\langle a^{2^i}, a^{j_2}x\rangle$, since $\langle a^{j_1-j_2}\rangle = \langle a^{2^i}\rangle$ for some $i$. Hence we need at most one generator of the form $a^jx$. Trivially, since $\<a \>$ is cyclic, we only need one generator of the form $a^{2^i}$ as well. 
 

Now consider the subgroup $\langle a^j x \rangle$. Its elements are $  \{a^jx, x^2=a^{2^k}, a^{2^k+j}x, 1\}$, so it  can be written in the form $\langle a^{2^k} ,a^j x \rangle$.


We conclude that subgroups of $DiC_{2^k}$ are of types 
\begin{enumerate}
\item $\langle a^{2^i}\rangle, ~0\le i \leq  k$
\item $\langle a^{2^i},a^jx\rangle, ~0\le i \leq k,~ 0\le j \le 2^i-1$.
\end{enumerate}

\begin{pro}
\label{pro:dic}
The dicyclic group $DiC_{2^{k-1}}$ is uniformly abelian group representable for all $n$. 
\end{pro}
\begin{proof}
This proof is an application of Lemma \ref{lem:psilemma}, which allows us to use the fact that $D_{2^k}$ was already shown to be abelian group representable, to conclude that so is $DiC_{2^{k-1}}$.

To apply Lemma \ref{lem:psilemma}, we must define a subgroup preserving bijection $$\psi: DiC_{2^{k-1}} \rightarrow D_{2^k}.$$

Since the elements in $DiC_{2^{k-1}}$ can be uniquely written as $$\{a^ix^j \st 0 \leq i < 2^k, j=0,1\}$$ while all the elements in $D_{2^k}$ can be uniquely written as $$\{r^is^j \st 0 \leq i < 2^k, j=0,1\},$$
 the map

$$\psi: a^ix^j \mapsto r^is^j$$ is well-defined and clearly bijective.
It remains to show that $\psi$ is subgroup preserving.
But from the discussion above, we see that the subgroups of $DiC_{2^{k-1}}$ are

\begin{enumerate}
\item $\langle a^{2^i}\rangle, ~0\le i < k$,
\item $\langle a^{2^i},a^jx\rangle, ~0\le i < k,~ 0\le j \le 2^i-1$.
\end{enumerate}

The images of both kinds of subgroups under $\psi$ do indeed form subgroups of $D_{2^k}$:
\begin{enumerate}
\item $\psi(\langle a^{2^i}\rangle) = \psi(\{    a^{2^ih} \st h \in \Z \}) =  \{r^{2^i h}  \st h \in \Z\} $,

\item $\psi(\langle a^{2^i},a^jx \rangle) = \psi (\{   a^{2^ih}, a^{2^i h +j }x \st h \in \Z \}) = \{r^{2^ih},r^{2^i h +j }s \st h \in \Z\}$.
\end{enumerate}

Hence $\psi$ is a subgroup preserving bijection and since by Proposition \ref{pro:dig} the dihedral group $D_{2^k}$ is abelian group representable, Lemma \ref{lem:psilemma} implies that $DiC_{2^{k-1}}$ is uniformly abelian group representable as well.

\end{proof}

Propositions \ref{pro:dig} and \ref{pro:dic}  generalize Proposition 3 of \cite{thom}.



The results we have obtained on representability of $2$-groups rely on the use of a subgroup preserving bijection $\psi$ and Lemma {\ref{lem:psilemma}}. The map $\psi$ defined for dihedral groups in Proposition \ref{pro:psidihedral} does not generalize to $p$-groups, as we can no longer trivially claim closure under inverses. We employ different methods to deal with abelian representability of odd order $p$-groups in the following section. 

%
%
\section{abelian Group Representability $p$-Groups}
\label{sec:p-groups}

We start with a simple lemma which establishes a necessary condition for abelian representability. In the next section this lemma is used to exclude the class of all non-nilpotent groups.
\begin{lem}
Let $G$ be a group which is abelian group representable. Then for any subgroups $G_1, G_2$ with intersection $G_{12}$ and corresponding indices $i_1, i_2, i_{12}$ in $G$, it must be the case that $$i_{12} \mid i_1i_2.$$
\label{lem:abrepIndices}

\end{lem}

\begin{proof}

Let $(A, A_1, A_2)$ represent $(G, G_1, G_2)$. Since $A$ is abelian,  the subgroups $A_1,A_2, A_{12}$ are all normal in $A$ and the quotient group $A/A_{12}$ is abelian of order $i_{12}$.

Next note that the subgroups $A_1/A_{12}$ and $A_2/A_{12}$ of $A/A_{12}$ are disjoint and normal. Hence $A/A_{12}$ contains the subgroup $(A_1/A_{12})(A_2/A_{12})$ whose order divides the order of $A/A_{12}$:  $$|A_1/A_{12}| | A_2/A_{12}|  \mid i_{12}.$$
Observing that $|A_{1}/A_{12}| = \frac{|A/{A_{12}}|}{|A/A_1|}$ we obtain

$$\frac{|A/{A_{12}|}}{|A/A_1|}  \frac{|A/{A_{12}|}}{|A/A_2|} = \frac{i_{12}}{i_1}\frac{i_{12}}{i_2}  \mid i_{12}.$$

Equivalently, $i_{12} |i_1i_2$.
\end{proof}

The following proposition proves abelian representability of $p$-groups for $n=2$ by  establishing a sufficient condition for $n=2$ and showing that all $p$-groups in fact satisfy it.

\begin{pro}
Let $G$ be a $p$-group. Then $G$ is uniformly abelian group representable for $n=2$.
\label{pro:abRepPgps}\end{pro}
 \begin{proof}
Let $p^m$ be the order of $G$. Consider some subgroups $G_1, G_2 , G_{12} = G_1\cap G_2$ of orders $p^i,p^j, p^k$ respectively. We show that the exponents $i,j,k,m$ obey an inequality, which is sufficient to guarantee abelian representability of $(G, G_1, G_2)$.

{\bf{Claim 1. Inequality  $i+j-k \le m$ holds for any $p$-group $G$.}}

To that end, consider the subset $S_{G_1}S_{G_2} = \{g_1g_2~|~g_1\in G_1, g_2 \in G_2 \}$. Note that $S_{G_1}S_{G_2}$ is only a subgroup when one of $G_1, G_2 $ is normal in $G$.
Counting the number of elements in $S_{G_1}S_{G_2}$, we have
$$|S_{G_1}S_{G_2}| = \frac{|G_1||G_2|}{|G_{12}|}= \frac{p^ip^j}{p^k} = p^{i+j-k}.$$
Now since $S_{G_1}S_{G_2}\subseteq G$, we conclude that $i+j-k \le m$.

{\bf{Claim 2. Sufficiency of condition $i+j-k \le m$.\\  }}
We show that $i+j-k \le m$ implies that $(G, G_1, G_2)$ can be represented by some abelian $(A, A_1, A_2)$.
Define $A$ to be the elementary abelian $p$-group $A = C_p^m$.

We can express $A$ as the following direct product:
$$C_p^k \times C_p^{i-k} \times C_p^{j-k} \times C_p^{m-(i+j-k)}.$$

Note that we needed the inequality $i+j-k \le m$ in order for the exponent $m - (i+j-k) $ to be nonnegative.
Now we define subgroups

$$A_1 =  C_p^k \times C_p^{i-k} \times \{1\} \times \{1\},$$
$$A_2 =  C_p^k \times \{1\} \times C_p^{j-k} \times \{1\}.$$

As the orders of $A, A_1, A_2, A_{12}$ are same as those of $G, G_1, G_2, G_{12}$, clearly $(A, A_1, A_2)$ represents $(G, G_1, G_2)$.

Since the choice of $G_1, G_2$ was arbitrary, we conclude that $G$ is abelian group representable for $n=2$.
Moreover, $G$ is uniformly abelian group representable, since $A$ was chosen independently of $G_1, G_2$.

\end{proof}

Claim 2 of the above proof establishes a sufficient condition for abelian group representability for $n=2$, namely the exponents $i,j,k, m$ of orders of $G_1, G_2, G_{12}, G$ obeying the inequality $i+j-k \le m$ . Lemma \ref{lem:abrepIndices} on the other hand implies that this condition is also necessary:
If $(G, G_1, G_2)$ is representable, then  $i_{12} \mid i_1i_2 $. But
$$i_{12} \mid i_1i_2 \iff \frac{ p^m}{p^k} \mid \frac{p^m}{p^{i}}  \frac{p^m}{p^{j}} \iff
\frac{p^{i} p^{j} }  {p^{k}} \mid p^m $$
$$\iff i+j-k\leq m.$$

Hence we have a numerical necessary and sufficient condition for abelian group representability of $(G, G_1, G_2)$ in terms of exponents $i,j, k, m$ of orders of subgroups $G_1, G_2, G_{12}, G$:
$$i+j-k \le m.$$
Certainly we can use similar methods as in the proof of Claim 2, Proposition \ref{pro:abRepPgps} to guarantee abelian representability whenever a similar  ``inclusion-exclusion" inequality holds for higher $n$. It will no longer, however, be  a necessary condition.
 More specifically for $G_1, G_2, G_3 \leq G$  of orders $|G_\Ac| = p^{j_\Ac} \st \Ac \subseteq\{1,2,3\}$, $ |G|= p^m$, the similar  inequality on exponents of orders
$$j_1+j_2 + j_3 - (j_{12}+j_{13}+j_{23}) +j_{123} \leq m $$
while guaranteeing representability of $(G, G_1, G_2, G_3)$, need no longer hold for groups which are abelian  representable. For example, we take the group of quaternions $Q_8$ with subgroups $\<i\>, \<j\>, \<k \>$, the exponents of whose orders violate the inequality for $n=3$: $2+2+2-(1+1+1)+1 =4 \not \leq 3$.

The following question arises:
\begin{rem}
Can we establish a  necessary and sufficient numerical condition for abelian group representability for $n>2$.

\end{rem}


We finish the section by giving a class of uniformly abelian representable $p$-groups for $n=3$.
\begin{pro}
If $p$ is a prime, a non-abelian group $G$ of order $p^3$ is abelian group representable for $n=3$.
\end{pro}

\begin{proof}
We show that the group $C_{p^2} \times C_p$ uniformly represents $G$. 
Since $G$ is a non-trivial $p$-group, it has a non-trivial center,~ $Z(G)$.\\ If $\mid Z(G) \mid = p^2$, $G/Z(G)$ is cyclic and hence $G$ is abelian. Thus we have $|Z(G)|=p$.

If $G_1,G_2$ and $G_3$ are three subgroups of $G$ of order $p^2$, then they are normal in $G$ having non-trivial intersection with $Z(G)$, implies that $Z(G)\subseteq G_1, G_2, G_3$ (Theorem 1, Chapter 6, \cite{DF}). So the pairwise intersection of all subgroups of $G$ of index $p$ is $Z(G)$.
Then the possible combinations of indices of $G_1, G_2, G_3, G_{12},G_{13},G_{23},G_{123}$ are

\begin{enumerate}
\item $p,p,p,p^2,p^2, p^2,p^2$
\item $p,p,p^2,p^2,p^2,p^2, p^2$
\item $p,p,p^2,p^2,p^2,p^3, p^3$
\item $p,p,p^2,p^2,p^3,p^3, p^3$
\item $p,p^2,p^2,p^2,p^2,p^3, p^3$
\item $p,p^2,p^2,p^2,p^3,p^3, p^3$
\item $p,p^2,p^2,p^3,p^3,p^3, p^3$
\item $p^2,p^2,p^2,p^3,p^3,p^3, p^3$
\end{enumerate}

The corresponding abelian group representation is given uniformly by
the group $A = C_{p^2} \times C_p$, where $C_{p^2} = \langle g|g^{p^2}=1\rangle,~C_p = \langle r|r^p=1\rangle $. For each combination of the indices above, let the subgroups $G_1, G_2, G_3$ be represented by $A_1, A_2, A_3$ respectively: 
\begin{enumerate}
\item $A_1=\langle g^p\rangle \times \langle r\rangle,~ A_2= \langle g\rangle \times 1 , ~A_3 = \langle(g,r)\rangle,~A_{12}= \langle g^p \rangle \times 1,~A_{13}= \langle g^p \rangle\times 1,~A_{23}=\langle g^p \rangle\times 1, ~A_{123}= \langle g^p \rangle\times 1$
\item $A_1=\langle g^p\rangle \times \langle r\rangle,~ A_2= \langle g\rangle \times 1 , ~A_3 = \langle g^p \rangle \times 1,~A_{12}= \langle g^p \rangle \times 1,~A_{13}= \langle g^p \rangle\times 1,~A_{23}=\langle g^p \rangle\times 1, ~A_{123}= \langle g^p \rangle\times 1$
\item $A_1=\langle g^p\rangle \times \langle r\rangle,~ A_2= \langle g\rangle \times 1 , ~A_3 = 1 \times \langle r \rangle,~A_{12}= \langle g^p \rangle \times 1,~A_{13}= 1 \times \langle r \rangle,~A_{23}=1 \times 1, ~ A_{123}= 1 \times 1$
\item $A_1= \langle (g,r)\rangle, ~  A_2= \langle g\rangle \times 1 , ~A_3 = 1 \times \langle r \rangle,~A_{12}= \langle g^p \rangle \times 1,~A_{13}= 1 \times 1,~A_{23}=1 \times 1, ~ A_{123}= 1 \times 1$
\item $A_1= \langle g^p\rangle \times \langle r\rangle, ~ A_2= \langle g^p \rangle \times 1,~A_3 = 1 \times \langle r \rangle,~ A_{12}=\langle g^p \rangle \times 1,~ A_{13}=1 \times \langle r \rangle,~ A_{23}=1 \times 1, ~ A_{123}=  1 \times 1$
\item $A_1= \langle g\rangle \times 1, ~ A_2= \langle g^p \rangle \times 1,~A_3 = \langle (g^p,r) \rangle,~A_{12}=\langle g^p \rangle \times 1,~A_{13}=1\times 1,~A_{23}=1 \times 1, ~ A_{123}=  1 \times 1$
\item $A_1= \langle(g,r)\rangle, ~ A_2= \langle(g^p,r)\rangle,~A_3 = 1 \times \langle r \rangle,~ A_{12}=1 \times 1,~A_{13}=1 \times 1,~A_{23}=1\times 1,~ A_{123}=  1 \times 1$
\item $A_1= \langle g^p \rangle \times 1, ~ A_2 = 1 \times \langle r \rangle,~A_3 = \langle (g^p,r)\rangle,~A_{12}=1 \times 1,~A_{13}=1 \times 1,~A_{23}=1 \times 1, ~ A_{123}=  1 \times 1$.
\end{enumerate}
We obtain that $A, A_1, A_2, A_3$ represents $G, G_1, G_2, G_3$ and since the group $A$ is independent of subgroups $G_1, G_2, G_3$, the representation of $G$ is uniform. 

\end{proof}
%
%

\section{abelian Group Representability of Nilpotent Groups}
\label{sec:nilpotent}
A finite nilpotent group is a direct product of its Sylow$_p$ subgroups. In this section we obtain a complete classification of abelian representable groups for $n=2$, as well as show that non nilpotent groups are never abelian group representable for general $n$. We start with an easy proposition which allows us to build new abelian representable groups from existing ones.

\begin{pro}
Let $G$ and $H$ be groups of coprime orders. Suppose $G, H$ are abelian group representable, then their direct product $G \times H$ is abelian group representable.
\end{pro}
\begin{proof}
Since the orders of $G, H$ are coprime,  any subgroup $K_i \leq G\times H$  is in fact a direct product $K_i = G_i \times H_i$, where $G_i \leq G, ~ H_i \leq H$. To prove abelian group representability for $n$, consider $n$ subgroups $G_1 \times H_1, \ldots , G_n \times H_n$ of $G\times H$.

Since $G$ is abelian group representable, there exists an abelian group $A$ with subgroups $A_1, \ldots, A_n$ which represents $(G, G_1, \ldots, G_n)$. Similarly some $(B, B_1, \ldots, B_n)$ represents $(H, H_1, \ldots, H_n)$.

It is easy to see that $(A \times B , A_1 \times B_1, \ldots, A_n \times B_n)$ represents $G_1 \times H_1, \ldots , G_n \times H_n$. To this end, consider an arbitrary intersection group $G_\Ac \times K_\Ac \leq G \times H$ for $\Ac \subseteq \{1, \ldots, n\}$:

$$[G\times K : (G_\Ac \times K_\Ac)] = [G: G_\Ac][H:H_\Ac] = $$
$$[A: A_\Ac][B:B_\Ac] = [A\times B : (A_\Ac \times B_\Ac)]. $$


\end{proof}

The previous proposition allows us to prove the following theorem.
\begin{thm}
All nilpotent groups are abelian group representable for $n=2$.
\end{thm}
\begin{proof}
We know that finite nilpotent groups are direct products of their Sylow$_p$ subgroups. That is, if $G$ is a finite nilpotent group, then
$G \cong S_{p_1} \times S_{p_2} \ldots \times S_{p_r}$  where $S_{p_i}$ is the Sylow$_{p_i}$ subgroup.
Since each $S_{p_i}$ is abelian group representable by Proposition \ref{pro:abRepPgps}, the previous result implies that $G$ is abelian group representable.
\end{proof}

Conversely, we show that the class of abelian representable groups must be contained inside the class of nilpotent groups for all $n$.

\begin{pro}
\label{pro:nonnil}
A non-nilpotent group $G$ is not abelian group representable for all values of $n$.
\end{pro}

\begin{proof}
Recall that, a group is nilpotent if and only if all of its Sylow subgroups are normal. Since $G$ is by assumption non-nilpotent, at least one of its Sylow$_p$ subgroup $S_p$ is not normal for a prime $p\mid r=|G|$. We know that $|S_p|=p^a$, the highest power of $p$ dividing $r$, that is, $r=p^am$ with $p\nmid m$. Since $S_p$ is not normal, $G$ has another Sylow$_p$ subgroup $S_p^x=\{xS_px^{-1}\mid x\in G, s\in S_p\}$ of order $p^a$ with $S_p\ne S_p^x$ for some $x$.

Say their intersection  $S_p\cap S_p^x$ is of order $p^t$. Since $S_p$ and $S_p^x$ are distinct, $a > t$ and thus $[G:S_p\cap S_p^x]=mp^{a-t}>m$. Thus we have subgroups $G_1=S_p, G_2=S_p^x, G_{12}= S_p\cap S_p^x$ of indices $m, m, mp^{a-t}$, respectively, in $G$. But this contradicts Lemma \ref{lem:abrepIndices}, which implies that $mp^{a-t} \mid m^2$.

\end{proof}

In particular, we have completely classified groups which are abelian group representable for $n=2$.
\begin{thm}A group $G$ is abelian group representable for $n=2$ if and only if it is nilpotent.
\end{thm}

The following corollary generalizes Proposition 4 of \cite{thom}.

\begin{cor}
\label{pro:digg}
If $m$ is not a power of $2$, then the dihedral group $D_m$, the quasi-dihedral groups $\qd_m$ and $\qdx_m$, and the dicyclic group $DiC_{m}$ are not abelian group representable for any $n>1$.
\end{cor}
\begin{proof}

Let $G_m$ denote either $D_m, \qd_m, \qdx_m$, or $DiC_m$. Since subgroups of nilpotent groups are nilpotent, it suffices to show that $G_p \leq G_m$ is not nilpotent for a prime $p \mid m$,  $p \geq 3$.  

In case when $G_m  = D_m, \qd_m, \qdx_m$ is (quasi-)dihedral, clearly $G_p$ is not nilpotent, as its Sylow$_2$ subgroup $\{1, s\}$ is not normal. Similarly, a Sylow$_2$ subgroup $\<x \> \leq DiC_p$ is not normal, which shows that $Dic_p$ is not nilpotent. 

By  Proposition \ref{pro:nonnil} the result follows.

\end{proof}

\section{Conclusions}
In this paper we propose a classification of finite groups with respect to the quasi-uniform variables induced by the subgroup structure. In particular, we study which finite groups belong to the same class as abelian groups with respect to this classification, that is, which finite groups can be represented by abelian groups. We provide an answer to this question when the number $n$ of quasi-uniform variables is 2: it is the class of nilpotent groups. For general $n$, we show  that nilpotent groups are abelian representable if and only  $p$-groups are, while non-nilpotent groups do not afford abelian representation. Hence the question of classifying abelian representable finite groups is completely reduced to answering the question for $p$-groups.
 
We demonstrate how some classes of $p$-groups afford abelian representation for all $n$, opening various interesting questions for further work. What other classes of $p$-groups can be shown to be abelian group representable? 
Is there a generalization of numerical criterion given for $n=2$ providing a necessary and sufficient condition for abelian representability? It would be extremely interesting to show whether $p$-groups are indeed abelian group representable. If not - what is the grading with respect to representability within $p$-groups (and, consequently, nilpotent groups)? Finally, beyond the nilpotent case, the classification of groups with respect to the quasi-uniform distributions is completely open, e.g. - what are the finite groups which induce the same quasi-uniform variables as solvable groups?


%
%

\end{document}